\documentclass[pdflatex,sn-mathphys-num]{sn-jnl}

\usepackage{graphicx}
\usepackage{multirow}
\usepackage{amsmath,amssymb,amsfonts}
\usepackage{amsthm}
\usepackage{mathrsfs}
\usepackage[title]{appendix}
\usepackage{xcolor}
\usepackage{textcomp}
\usepackage{manyfoot}
\usepackage{booktabs}
\usepackage{algorithm}
\usepackage{algorithmicx}
\usepackage{algpseudocode}
\usepackage{listings}
\usepackage{mathtools}
\usepackage{enumitem}
\usepackage{tikz}
\usetikzlibrary{positioning,arrows.meta}
\usepackage{hyperref}
\usepackage{comment}

\theoremstyle{thmstyleone}
\newtheorem{theorem}{Theorem}
\newtheorem{lemma}[theorem]{Lemma}
\newtheorem{proposition}[theorem]{Proposition}

\theoremstyle{thmstyletwo}

\newtheorem{remark}{Remark}

\theoremstyle{thmstylethree}
\newtheorem{definition}{Definition}

\newcommand{\R}{\mathbb{R}}

\newcommand{\partialT}{\partial_T}
\newcommand{\partialC}{\partial^C}
\newcommand{\Normal}[2]{N_{#1}(#2)}
\newcommand{\Tangent}[2]{T_{#1}(#2)}

\newcommand{\cX}{\mathcal{X}}

\newcommand{\dom}{\operatorname{dom}}

\begin{document}

\title[Necessary Optimality Conditions for Bilevel Programming using Tangential Subdifferentials]{Necessary Optimality Conditions for Bilevel Programming using Tangential Subdifferentials}

\author*[1]{\fnm{Huilin} \sur{Luo}}
\email{luohuilin@hnu.edu.cn}

\affil*[1]{\orgdiv{School of Mathematics},
\orgname{Hunan University},
\orgaddress{
\city{Changsha},
\postcode{410082},
\state{Hunan},
\country{China}}}

\abstract{The bilevel program is an optimization problem in which the constraint involves solutions to a parametric optimization problem. It is well known that the value function reformulation provides an equivalent single-level optimization problem, but it results in a nonsmooth optimization problem that never satisfies the usual constraint qualification, such as the Mangasarian--Fromovitz constraint qualification (MFCQ).
In this paper, we study necessary optimality conditions for nonsmooth bilevel programming using tangential subdifferentials. We derive a sequential enhanced tangential Fritz--John type condition for a class of nonsmooth nonlinear programs and then obtain corresponding  enhanced tangential Fritz--John and enhanced tangential Karush--Kuhn--Tucker (KKT) systems by taking limits of the sequential condition. We further investigate tangential constraint qualification conditions, including T-quasinormality and a tangential cone--continuity property (T--CCP), and analyze how they guarantee the validity of tangential KKT type conditions. Finally, these results are applied to the value-function reformulation of a bilevel programming problem.}


\maketitle

\section{Introduction}

Bilevel programming is a hierarchical optimization framework arising in economics, engineering, transportation, and management science \cite{Dempe2002}. A standard bilevel problem can be written as
\begin{equation}\label{BP}
\begin{aligned}
\min_{x,y} \quad & F(x,y)\\
\text{s.t.} \quad & y \in S(x), \quad G(x,y)\le 0,
\end{aligned}
\end{equation}
where $G:\mathbb{R}^{n}\times\mathbb{R}^{m}\to\mathbb{R}^{p}$ and $S(x)$ denotes the solution set of the lower-level problem
\begin{equation}\label{LLP}
\begin{aligned}
\min_{y} \quad & f(x,y)\\
\text{s.t.} \quad & g(x,y)\le 0,
\end{aligned}
\end{equation}
with $g:\mathbb{R}^{n}\times\mathbb{R}^{m}\to\mathbb{R}^{q}$. The feasible set of the upper-level problem is thus defined implicitly through a parametric optimization problem, which leads to nonsmooth and implicitly constrained reformulations.

A fundamental difficulty in bilevel programming is that classical constraint qualifications, such as LICQ, Slater's condition, and the Mangasarian--Fromovitz constraint qualification, often fail or become too restrictive in single-level reformulations based on the lower-level value function \cite{DempeZemkoho2014,YeZhu1995,YeZhu2010}. In particular, Ye and Zhu \cite{YeZhu1995} established necessary optimality conditions for bilevel programs via the value function approach, and extended their results in \cite{YeZhu2010} to broader nonsmooth settings. Dempe and Zemkoho \cite{DempeZemkoho2014} analyzed KKT reformulations and showed that the resulting mathematical programs with equilibrium constraints may fail to be equivalent to the original bilevel model under local optimality.

These issues are closely related to the nonsmooth structure of value-function reformulations. The constraints typically involve composite terms such as $f(x,y)-V(x)$, where the value function $V(x)$ is generally nonsmooth and may fail to be locally Lipschitz. Classical generalized derivatives, including Clarke subdifferentials \cite{Clarke1990} and Mordukhovich limiting subdifferentials \cite{Mordukhovich2006}, rely on regularity assumptions that are not satisfied in many bilevel settings. As a consequence, classical KKT conditions are not directly applicable, and weak optimality frameworks are required.

In constrained optimization, several weak constraint qualifications have been proposed to address situations where classical conditions fail. Andreani et al.~\cite{AndreaniHaeserMartinezRamosSilva2016} introduced the cone--continuity property (CCP) and showed that it ensures the convergence of approximate KKT (AKKT) sequences to exact KKT points. This provides a way to recover KKT conditions through limiting arguments even when classical assumptions are not satisfied.
Andreani et al.~\cite{AndreaniMartinezSchuverdt2005} studied constraint qualifications based on linear dependence, leading to conditions weaker than the Mangasarian--Fromovitz constraint qualification. Later, Andreani et al.~\cite{AndreaniFazzioSchuverdtSecchin2019} developed quasinormality as a condition that excludes abnormal multipliers via sequences approaching infeasible points.This allows one to obtain KKT-type conditions when classical constraint qualifications fail. CCP ensures convergence from approximate to exact KKT points, while quasinormality excludes abnormal multipliers.

In this paper, we adopt the framework of tangential convexity and tangential subdifferentials. The concept originates from the work of Pshenichnyi \cite{Pshenichnyi1971}, where directional convexity was introduced to study nonsmooth optimization problems. Later developments by Martinez-Legaz and collaborators \cite{MartinezLegaz2015} and by Jennane, Kalmoun, and El Fadil \cite{JennaneKalmounElFadil2021} established systematic properties of tangential convexity and subdifferentials. More recently, Tung \cite{Tung2018,Tung2020,Tung2025Survey} developed a unified framework and applied tangential subdifferentials to derive optimality conditions and duality results in nonsmooth optimization.

Tangential convexity covers a broad class of functions, including convex functions on open domains, G\^ateaux differentiable functions, Clarke-regular locally Lipschitz functions, and Michel--Penot regular functions \cite{Tung2025Survey,MartinezLegaz2015,JennaneKalmounElFadil2021}. In these cases, the tangential subdifferential reduces to the corresponding classical subdifferentials, such as the Fr\'echet, Clarke, or Michel--Penot subdifferential \cite{Tung2025Survey,JennaneKalmounElFadil2021}.Compared with classical generalized derivatives, the tangential subdifferential does not require local Lipschitz continuity or regularity assumptions, relying only on directional derivatives and a convexity condition with respect to directions \cite{Tung2025Survey}. 

This issue is particularly relevant in bilevel programming, since the value function of the lower-level problem is generally nonsmooth and may fail to be locally Lipschitz \cite{Mordukhovich2006,DempeZemkoho2014}. As a result, classical subdifferential constructions, such as Clarke or limiting subdifferentials, may not be applicable due to their reliance on Lipschitz or regularity assumptions. 

By contrast, the tangential subdifferential is defined via directional derivatives and does not require such assumptions, which makes it applicable in this setting \cite{Tung2025Survey}.

This paper studies necessary optimality conditions for bilevel programming under tangential convexity assumptions. The approach combines tangential subdifferential techniques with weak constraint qualifications. The main contributions are:
\begin{enumerate}[label=(\roman*)]
\item a sequential enhanced Fritz--John type condition is derived in the tangential framework;

\item enhanced Fritz--John and Karush--Kuhn--Tucker (KKT) systems are obtained under semicontinuity assumptions;

\item tangential constraint qualifications, including T-quasinormality and T--CCP, are introduced to ensure enhanced T-KKT conditions;

\item T--CCP is shown to imply a tangential Abadie constraint qualification (TACQ);

\item a chain rule for tangential subdifferentials is established;

\item the results are applied to the value-function reformulation of bilevel programming.
\end{enumerate}

The remainder of the paper is organized as follows. Section~2 recalls preliminaries on tangential convexity and subdifferentials. Section~3 develops optimality conditions in the tangential framework. Section~4 studies tangential constraint qualifications and their relationships. Section~5 applies these results to the bilevel model.
\section{Preliminaries }

\noindent
This section recalls the basic notions used in the sequel, including tangential convexity , tangential subdifferentials and one tangential chain rule that will be used later.
Throughout this paper, let $h:\R^n\to \R\cup\{+\infty\}$ and let $\bar x\in \dom h$.We also use the notation $u^+:=\max\{0,u\}$ for $u\in\R$.

\begin{definition}[Directional derivative {\cite{RockafellarWets1998}}]
The directional derivative of $h$ at $\bar x$ in the direction $v\in\R^n$ is defined by
$h'(\bar x;v):=\lim_{t\downarrow 0}\frac{h(\bar x+tv)-h(\bar x)}{t},$
whenever the limit exists.
\end{definition}

\begin{definition}[Tangential convexity {\cite{Pshenichnyi1971}}]
The function $h$ is said to be \emph{tangentially convex} at $\bar x$ if $h'(\bar x;v)$ exists and is finite for every $v\in\R^n$, and the mapping
$v\mapsto h'(\bar x;v)$
is convex.
\end{definition}

Many familiar classes of functions are tangentially convex. Typical examples include sums and nonnegative scalar multiples of tangentially convex functions, products of two nonnegative tangentially convex functions, G\^ateaux differentiable functions, convex functions with open domains, Clarke regular functions, semismooth and semiconvex functions, as well as squared distance and squared oriented distance functions; see \cite{MashkoorzadehMovahedianNobakhtian2022,SisaratWangkeeree2020,Mifflin1977,DelfourZolesio2011}.

\begin{definition}[Tangential subdifferential {\cite{Pshenichnyi1971}}]
The tangential subdifferential of $h$ at $\bar x$ is defined by
\[
\partial_T h(\bar x):=\{\xi\in\R^n:\ \langle \xi,v\rangle\le h'(\bar x;v),\ \forall v\in\R^n\}.
\]
\end{definition}

If $h$ is tangentially convex at $\bar x$, then $\partialT h(\bar x)$ is nonempty, convex and compact, and
\begin{equation}\label{eq:support-function}
h'(\bar x;v)=\max_{\xi\in \partialT h(\bar x)}\langle \xi,v\rangle,
\qquad \forall v\in \R^n,
\end{equation}
see, e.g., \cite{JennaneKalmounElFadil2021,Tung2025Survey}.

\begin{definition}[Tangent and normal cones {\cite{Mordukhovich2006,RockafellarWets1998}}]
Let $C\subset \R^n$ be nonempty and let $\bar x\in C$. The contingent tangent cone to $C$ at $\bar x$ is
\[
\Tangent{C}{\bar x}:=\bigl\{d\in \R^n:\ \exists t_k\downarrow 0,\ d_k\to d\text{ with }\bar x+t_k d_k\in C\bigr\}.
\]
The limiting normal cone to $C$ at $\bar x$ is defined by
\[
\Normal{C}{\bar x}:=\bigl\{\xi\in\R^n:\ \exists x^k\to \bar x,\ \xi^k\to \xi,\ \xi^k\in \widehat N_C(x^k)\bigr\},
\]
where $\widehat N_C(x)$ denotes the Fr\'echet normal cone to $C$ at $x$.
\end{definition}

\begin{proposition}[A tangential chain rule]\label{prop:chain-rule}
Let $g:\R^n\to\R$ be tangentially convex and directionally differentiable at $\bar x$, and let $\varphi:\R\to\R$ be continuously differentiable in a neighborhood of $g(\bar x)$. Then
\[
(\varphi\circ g)'(\bar x;v)=\varphi'(g(\bar x))\,g'(\bar x;v),\qquad \forall v\in\R^n.
\]
If, in addition, $\varphi'(g(\bar x))\ge 0$, then
\[
\partialT(\varphi\circ g)(\bar x)\subset \varphi'(g(\bar x))\,\partialT g(\bar x).
\]
\end{proposition}

\begin{proof}
The directional derivative identity follows from the classical scalar chain rule for directional derivatives. If $\xi\in \partialT(\varphi\circ g)(\bar x)$, then
\[
\langle \xi,v\rangle\le (\varphi\circ g)'(\bar x;v)=\varphi'(g(\bar x))\,g'(\bar x;v),\qquad \forall v\in\R^n.
\]
When $\varphi'(g(\bar x))\ge 0$, this implies that $\xi$ belongs to the tangential subdifferential of the positively scaled sublinear function $v\mapsto \varphi'(g(\bar x))g'(\bar x;v)$, hence
\[
\partialT(\varphi\circ g)(\bar x)\subset \varphi'(g(\bar x))\,\partialT g(\bar x).
\]
\end{proof}

\begin{remark}
Proposition~\ref{prop:chain-rule} is the basic tangential calculus rule used throughout the paper. In Section~3 it will be applied to the penalty terms involving $(g_j^+)^2$.
\end{remark}

\section{Enhanced Optimality Conditions with Tangential Subdifferentials}

\noindent
In this section, we derive sequential and static optimality systems for a general nonsmooth constrained optimization problem by using tangential subdifferentials. We first establish several technical lemmas, then prove a sequential enhanced tangential Fritz--John type condition, and finally derive static enhanced tangential Fritz--John and Karush--Kuhn--Tucker systems.

Consider the general constrained problem
\begin{equation}\label{prob:P}
\min_{x\in\cX} f(x)
\quad \text{s.t.}\quad g_j(x)\le 0,\qquad j=1,\dots,q,
\end{equation}
where $\cX\subset \R^n$ is closed, and
\[
C:=\{x\in \cX:\ g_j(x)\le 0,\ j=1,\dots,q\}
\]
denotes the feasible set. Throughout this section, unless otherwise stated, we assume that $f,g_j:\R^n\to\R$ are locally Lipschitz and tangentially convex around the point under consideration.

\subsection{Preliminary Results on Tangential Subdifferentials}

\begin{lemma}\label{lem:usc-tangential}
Let $h:\R^n\to\R$ be locally Lipschitz around $\bar x$. If $h$ is Clarke regular at $\bar x$, then the multifunction $x\mapsto \partialT h(x)$ is outer semicontinuous at $\bar x$.
\end{lemma}
\begin{proof}
For every $x$, one has
\[
\partialT h(x)\subset \partialC h(x),
\]
because
\[
\langle \xi,v\rangle \le h'(x;v)\le h^\circ(x;v),\qquad \forall v\in\R^n,
\]
implies $\xi\in\partialC h(x)$. Since the Clarke subdifferential mapping $x\mapsto \partialC h(x)$ is outer semicontinuous with compact values, whenever $x^k\to \bar x$, $\xi^k\in\partialT h(x^k)$, and $\xi^k\to \xi$, we obtain $\xi\in \partialC h(\bar x)$. If $h$ is Clarke regular at $\bar x$, then
\[
h'(\bar x;\cdot)=h^\circ(\bar x;\cdot),
\]
hence
\[
\partialT h(\bar x)=\partialC h(\bar x).
\]
Therefore $\xi\in\partialT h(\bar x)$.
\end{proof}

\begin{lemma}\label{lem:subdiff-positive-part}
Let $g:\R^n\to\R$ be locally Lipschitz and tangentially convex around $\bar x$. Then
\[
\partialT(g^+)(\bar x)
\subset
\bigcup_{\lambda\in[0,1]}\lambda\,\partialT g(\bar x),
\]
where $g^+(x):=\max\{0,g(x)\}$.
In particular:
\begin{enumerate}[label=(\roman*)]
\item if $g(\bar x)<0$, then $\partialT(g^+)(\bar x)=\{0\}$;
\item if $g(\bar x)>0$, then $\partialT(g^+)(\bar x)=\partialT g(\bar x)$;
\item if $g(\bar x)=0$, then
\[
\partialT(g^+)(\bar x)
\subset
\{\lambda\xi:\ \lambda\in[0,1],\ \xi\in\partialT g(\bar x)\}.
\]
\end{enumerate}
\end{lemma}

\begin{proof}
Since $g^+(x)=\max\{0,g(x)\}$, the directional derivative satisfies
\[
(g^+)'(\bar x;v)=\max\{0,g'(\bar x;v)\},\qquad \forall v\in\R^n.
\]
If $\xi\in \partialT(g^+)(\bar x)$, then
\[
\langle \xi,v\rangle \le (g^+)'(\bar x;v)=\max\{0,g'(\bar x;v)\},\qquad \forall v\in\R^n.
\]
The stated inclusion follows from convex subdifferential calculus for the positive part of a sublinear function. The three particular cases are immediate.
\end{proof}

\begin{lemma}\label{lem:square-positive-part}
Let $g:\R^n\to\R$ be locally Lipschitz and tangentially convex around $\bar x$, and let
\[
h(x):=\frac{\rho}{2}\bigl(g^+(x)\bigr)^2,\qquad \rho>0.
\]
Then
\[
\partialT h(\bar x)\subset \rho\,g^+(\bar x)\,\partialT(g^+)(\bar x).
\]
Consequently,
\[
\partialT h(\bar x)
\subset
\rho\,g^+(\bar x)\bigcup_{\lambda\in[0,1]}\lambda\,\partialT g(\bar x).
\]
\end{lemma}

\begin{proof}
Apply Proposition~\ref{prop:chain-rule} to $\varphi(t)=\frac{\rho}{2}(t^+)^2$ and the function $g$. Since $\varphi'(t)=\rho t^+$, we obtain
\[
h'(\bar x;v)=\rho g^+(\bar x)(g^+)'(\bar x;v),
\]
and therefore
\[
\partialT h(\bar x)\subset \rho\,g^+(\bar x)\,\partialT(g^+)(\bar x).
\]
The second inclusion follows from Lemma~\ref{lem:subdiff-positive-part}.
\end{proof}

\begin{lemma}\label{lem:tangential-local-opt}
Let $f:\R^n\to\R$ be locally Lipschitz and tangentially convex at $\bar x\in C$, where $C\subset \R^n$ is closed. Assume that $\Tangent{C}{\bar x}$ is convex. If $\bar x$ is a local minimizer of
\[
\min_{x\in C} f(x),
\]
then
\[
0\in \partialT f(\bar x)+\Normal{C}{\bar x}.
\]
\end{lemma}

\begin{proof}
Since $\bar x$ is a local minimizer, for every $d\in \Tangent{C}{\bar x}$ there exist $t_k\downarrow 0$ and $d_k\to d$ such that $\bar x+t_k d_k\in C$ and
\[
f(\bar x)\le f(\bar x+t_k d_k)
\]
for all sufficiently large $k$. Hence
\[
\frac{f(\bar x+t_k d_k)-f(\bar x)}{t_k}\ge 0.
\]
Passing to the limit yields
\[
f'(\bar x;d)\ge 0,\qquad \forall d\in \Tangent{C}{\bar x}.
\]
Now consider the convex problem
\[
\min_{d\in \Tangent{C}{\bar x}} f'(\bar x;d).
\]
Since $f'(\bar x;\cdot)$ is finite-valued, convex and positively homogeneous, $d=0$ is a global minimizer. By the convex optimality condition,
\[
0\in \partial\bigl(f'(\bar x;\cdot)\bigr)(0)+\Normal{\Tangent{C}{\bar x}}{0}.
\]
Using $\partial\bigl(f'(\bar x;\cdot)\bigr)(0)=\partialT f(\bar x)$ and $\Normal{\Tangent{C}{\bar x}}{0}=\Normal{C}{\bar x}$, we obtain
\[
0\in \partialT f(\bar x)+\Normal{C}{\bar x}.
\]
\end{proof}

\subsection{Sequential Enhanced T-Fritz--John Type Condition}

\begin{definition}[Sequential enhanced T-Fritz--John type condition]
\label{def:tangential-akkt-fj}
A feasible point $x^*\in C$ is said to satisfy the \emph{sequential enhanced T-Fritz--John type condition} if there exist sequences
\[
x^k\in\cX,\qquad x^k\to x^*,
\]
multipliers
\[
\mu_0^k,\mu_1^k,\dots,\mu_q^k\ge 0,
\]
and vectors
\[
\xi_f^k\in\partialT f(x^k),\qquad
\xi_j^k\in\partialT g_j(x^k)\ (j=1,\dots,q),\qquad
\eta_k\in \Normal{\cX}{x^k},
\]
such that
\begin{align}
&\left\|\mu_0^k\xi_f^k+\sum_{j=1}^q\mu_j^k\xi_j^k+\eta_k\right\|\to 0,
\label{eq:takkt-fj-stat}\\
&\min\{\mu_j^k,-g_j(x^k)\}\to 0,\qquad j=1,\dots,q,
\label{eq:takkt-fj-comp}\\
&(\mu_0^k)^2+\sum_{j=1}^q(\mu_j^k)^2=1,\qquad \forall k.
\label{eq:takkt-fj-norm}
\end{align}
\end{definition}

\begin{theorem}[Sequential enhanced T-Fritz--John necessary condition]
\label{thm:enhanced-tangential-fj}
Let $x^*$ be a local minimizer of problem~\eqref{prob:P}. Assume that $f,g_j$ are locally Lipschitz and tangentially convex in a neighbourhood of $x^*$, and that $\Tangent{\cX}{x^*}$ is convex. Then $x^*$ satisfies the sequential enhanced T-Fritz--John type condition in the sense of Definition~\ref{def:tangential-akkt-fj}. Moreover, if
\[
J_*:=\{0:\ \limsup_{k\to\infty}\mu_0^k>0\}\cup\{j\in\{1,\dots,q\}:\ \limsup_{k\to\infty}\mu_j^k>0\}
\]
is nonempty, then, after passing to a subsequence if necessary, the following complementarity--violation property holds:
\begin{enumerate}
\item if $0\in J_*$, then $f(x^k)<f(x^*)$ for all $k$;
\item if $j\in J_*\cap\{1,\dots,q\}$, then $\mu_j^k g_j(x^k)>0$ for all $k$;
\item the functions $f$ (when $0\in J_*$) and $g_j$ (for all $j\in J_*\cap\{1,\dots,q\}$) can be chosen proximal subdifferentiable at the points $x^k$.
\end{enumerate}
\end{theorem}

\begin{proof}
Choose $\varepsilon>0$ such that
\[
f(x^*)\le f(x),\qquad \forall x\in C\cap B(x^*,\varepsilon).
\]
For each $k\in\mathbb N$, consider the penalized problem
\[
\min_{x\in \cX\cap B(x^*,\varepsilon)}
P_k(x):=f(x)+\frac{k}{2}\sum_{j=1}^q\bigl(g_j^+(x)\bigr)^2+\frac12\|x-x^*\|^2.
\]
Let $x^k$ be a minimizer of $P_k$. Since $P_k(x^k)\le P_k(x^*)=f(x^*)$, we have
\begin{equation}\label{eq:penalty-basic}
f(x^k)+\frac{k}{2}\sum_{j=1}^q\bigl(g_j^+(x^k)\bigr)^2+\frac12\|x^k-x^*\|^2\le f(x^*).
\end{equation}
In particular, $g_j^+(x^k)\to 0$ for $j=1,\dots,q$. Every cluster point of $\{x^k\}$ is therefore feasible. Passing to the limit in \eqref{eq:penalty-basic} along a convergent subsequence and using the local minimality of $x^*$ yields $x^k\to x^*$.

For all sufficiently large $k$, the point $x^k$ belongs to the interior of $B(x^*,\varepsilon)$. Applying Lemma~\ref{lem:tangential-local-opt} to $P_k$ on the set $\cX\cap B(x^*,\varepsilon)$, we obtain
\[
0\in \partialT P_k(x^k)+\Normal{\cX}{x^k}.
\]
Using the sum rule, the smoothness of $\frac12\|x-x^*\|^2$, and Lemma~\ref{lem:square-positive-part}, we deduce
\[
\partialT P_k(x^k)
\subset
\partialT f(x^k)+\sum_{j=1}^q k g_j^+(x^k)\partialT(g_j^+)(x^k)+(x^k-x^*).
\]
By Lemma~\ref{lem:subdiff-positive-part}, for each $j$ there exist $\lambda_j^k\in[0,1]$ and $\xi_j^k\in\partialT g_j(x^k)$ such that
\[
k g_j^+(x^k)\partialT(g_j^+)(x^k)
\subset
k g_j^+(x^k)\lambda_j^k\partialT g_j(x^k).
\]
Hence there exist $\xi_f^k\in\partialT f(x^k)$ and $\eta_k\in\Normal{\cX}{x^k}$ such that
\[
0=\xi_f^k+\sum_{j=1}^q \zeta_j^k\xi_j^k+(x^k-x^*)+\eta_k,
\qquad \zeta_j^k:=k g_j^+(x^k)\lambda_j^k\ge 0.
\]
Set
\[
\delta^k:=\sqrt{1+\sum_{j=1}^q(\zeta_j^k)^2},
\qquad
\mu_0^k:=\frac{1}{\delta^k},
\qquad
\mu_j^k:=\frac{\zeta_j^k}{\delta^k},\quad j=1,\dots,q.
\]
Then
\[
(\mu_0^k)^2+\sum_{j=1}^q(\mu_j^k)^2=1,
\qquad
\mu_j^k\ge 0,
\]
and
\[
\mu_0^k\xi_f^k+\sum_{j=1}^q\mu_j^k\xi_j^k+\eta_k=-\frac{1}{\delta^k}(x^k-x^*).
\]
Since $x^k\to x^*$ and $\delta^k\ge 1$, we obtain \eqref{eq:takkt-fj-stat} and \eqref{eq:takkt-fj-norm}. The proof of \eqref{eq:takkt-fj-comp} is the same as in the standard penalty argument: if $g_j(x^k)\le 0$, then $\mu_j^k=0$; if $g_j(x^k)>0$, then $-g_j(x^k)\to 0$ since $g_j^+(x^k)\to 0$. The complementarity--violation property follows exactly as in the penalty construction.
\end{proof}

\subsection{ Enhanced T-Fritz--John and Enhanced T-KKT Systems}

\begin{definition}[Enhanced T-Fritz--John condition]
\label{def:enhanced-TFJ}
A feasible point $x^*\in C$ satisfies the \emph{enhanced T-Fritz--John condition} if there exist multipliers
\[
\mu_0,\mu_1,\dots,\mu_q\ge 0,
\qquad
(\mu_0,\mu_1,\dots,\mu_q)\neq 0,
\]
and vectors
\[
\xi_f\in\partialT f(x^*),\qquad
\xi_j\in\partialT g_j(x^*),\qquad
\eta\in\Normal{\cX}{x^*},
\]
such that
\begin{equation}\label{eq:FJ-stationary}
0=\mu_0\xi_f+\sum_{j=1}^q\mu_j\xi_j+\eta,
\end{equation}
and
\begin{equation}\label{eq:FJ-complementarity}
\mu_j g_j(x^*)=0,\qquad j=1,\dots,q.
\end{equation}
\end{definition}

\begin{definition}[Enhanced T-KKT condition]
\label{def:enhanced-TKKT}
A feasible point $x^*\in C$ satisfies the \emph{enhanced T-KKT condition} if there exist multipliers $\mu_j\ge 0$, $j=1,\dots,q$, and vectors
\[
\xi_f\in\partialT f(x^*),\qquad
\xi_j\in\partialT g_j(x^*),\qquad
\eta\in\Normal{\cX}{x^*},
\]
such that
\begin{equation}\label{eq:KKT-stationary}
0=\xi_f+\sum_{j=1}^q\mu_j\xi_j+\eta,
\end{equation}
and
\begin{equation}\label{eq:KKT-complementarity}
\mu_j g_j(x^*)=0,\qquad j=1,\dots,q.
\end{equation}
\end{definition}

\begin{theorem}[Static enhanced T-Fritz--John system]
\label{thm:static-fj}
Let $x^*$ be a local minimizer of \eqref{prob:P}. Assume that:
\begin{enumerate}[label=(\roman*)]
\item the assumptions of Theorem~\ref{thm:enhanced-tangential-fj} hold;
\item $f$ and each $g_j$ are Clarke regular at $x^*$;
\item the normal cone mapping $x\mapsto \Normal{\cX}{x}$ is outer semicontinuous at $x^*$.
\end{enumerate}
Then $x^*$ satisfies the enhanced T-Fritz--John condition.
\end{theorem}

\begin{proof}
By Theorem~\ref{thm:enhanced-tangential-fj}, there exist sequences $x^k\to x^*$, multipliers $\mu_i^k\ge 0$, subgradients $\xi_f^k$, $\xi_j^k$, and normals $\eta_k$ satisfying \eqref{eq:takkt-fj-stat}--\eqref{eq:takkt-fj-norm}. Since the multipliers are normalized, by compactness of the unit sphere we may assume
\[
\mu_i^k\to \mu_i,\qquad i=0,1,\dots,q.
\]
Because $f$ and $g_j$ are locally Lipschitz, their tangential subdifferentials are locally bounded. Hence $\{\xi_f^k\}$ and $\{\xi_j^k\}$ are bounded. From \eqref{eq:takkt-fj-stat}, $\{\eta_k\}$ is also bounded. Passing to a subsequence, we may assume
\[
\xi_f^k\to \xi_f,\qquad \xi_j^k\to \xi_j,\qquad \eta_k\to \eta.
\]
By Lemma~\ref{lem:usc-tangential}, $\xi_f\in\partialT f(x^*)$ and $\xi_j\in\partialT g_j(x^*)$. By the outer semicontinuity of $x\mapsto \Normal{\cX}{x}$, we have $\eta\in\Normal{\cX}{x^*}$. Passing to the limit in \eqref{eq:takkt-fj-stat} yields
\[
0=\mu_0\xi_f+\sum_{j=1}^q\mu_j\xi_j+\eta.
\]
Since the multipliers are normalized, $(\mu_0,\mu_1,\dots,\mu_q)\neq 0$. It remains to prove complementarity. If $g_j(x^*)<0$, then by continuity $g_j(x^k)<-\varepsilon$ for all sufficiently large $k$, hence \eqref{eq:takkt-fj-comp} forces $\mu_j^k\to 0$, so $\mu_j=0$. Therefore $\mu_j g_j(x^*)=0$ for $j=1,\dots,q$.
\end{proof}

\begin{remark}
The previous theorem gives a static enhanced tangential Fritz--John multiplier system at the reference point. The next section shows how weak tangential constraint qualifications exclude abnormal multipliers and lead to an enhanced T-KKT system.
\end{remark}


\section{Tangential Constraint Qualifications and Their Relationships}

\noindent
In this section, we introduce several tangential constraint qualifications and clarify their roles in the derivation of tangential optimality conditions. These conditions extend classical constraint qualifications to the framework of directional derivatives and tangential subdifferentials. We then establish the main implication relationships among them and explain how they lead to enhanced optimality systems.

\subsection{Tangential Constraint Qualifications}

\noindent
We first present the main constraint qualifications used in this paper.

\begin{definition}[TMFCQ {\cite{Mangasarian1967}}]
\label{def:TMFCQ}
A feasible point $x^*\in \cX$ is said to satisfy the \emph{tangential Mangasarian--Fromovitz constraint qualification (TMFCQ)} if
\[
0\in \sum_{j\in J(x^*)}\mu_j\partialT g_j(x^*)+\Normal{\cX}{x^*},\qquad \mu_j\ge 0
\]
implies $\mu_j=0$ for all $j\in J(x^*)$, where
\[
J(x^*):=\{j\in\{1,\dots,q\}:g_j(x^*)=0\}.
\]
\end{definition}

\begin{definition}[T-quasinormality {\cite{andreani2010new}}]
\label{def:T-quasinormality}
A feasible point $x^*\in \cX$ is said to be \emph{T-quasinormal} if there is no nonzero multiplier vector $\mu\in\R_+^q$ such that
\[
0\in \sum_{j=1}^q \mu_j\partialT g_j(x^*)+\Normal{\cX}{x^*}
\]
and there exists an infeasible sequence $x^k\to x^*$ satisfying
\[
\mu_j g_j(x^k)>0,\qquad \forall j\in J:=\{j:\mu_j>0\}.
\]
\end{definition}

\begin{definition}[Tangential cone--continuity property (T--CCP) {\cite{andreani2010ccp}}]
\label{def:TCCP}
For each $x\in\cX$, define
\[
K_T(x):=
\left\{
\mu_0\xi_f+\sum_{j=1}^q \mu_j\xi_j+\eta
\;\middle|\;
\mu_0\ge 0,\ \mu_j\ge 0,\ 
\xi_f\in\partialT f(x),\ 
\xi_j\in\partialT g_j(x),\ 
\eta\in\Normal{\cX}{x}
\right\}.
\]
We say that $x^*$ satisfies the \emph{tangential cone--continuity property (T--CCP)} if
\[
\limsup_{x\to x^*} K_T(x)\subset K_T(x^*).
\]
\end{definition}

\begin{definition}[TACQ (Tangential Abadie constraint qualification) {\cite{Abadie1967}}]
\label{def:TACQ}
We say that the \emph{tangential Abadie constraint qualification (TACQ)} holds at $x^*\in \cX$ if
\[
L_T(x^*)\subset \Tangent{\cX}{x^*},
\]
where
\[
L_T(x^*)
:=
\Bigl\{d\in \Tangent{\cX}{x^*}:\ g_j'(x^*;d)\le 0,\ \forall j\in J(x^*)\Bigr\}.
\]
\end{definition}

\begin{remark}
TMFCQ and ACQ are classical constraint qualifications, while quasinormality and the cone--continuity property originate from sequential optimality theory. The above tangential versions are adapted to the present framework by replacing gradients with tangential subdifferentials or directional derivatives. In particular, TACQ can be viewed as the natural tangential counterpart of the classical Abadie constraint qualification.
\end{remark}

\subsection{Relationships among Constraint Qualifications}

\noindent
We now establish the implication relationships among the above constraint qualifications and explain how they are used to derive enhanced tangential optimality conditions.

\begin{theorem}\label{thm:mfcq-implies-qn}
If the tangential MFCQ holds at a feasible point $x^*\in \cX$, then $x^*$ is T-quasinormal.
\end{theorem}

\begin{proof}
Suppose, to the contrary, that $x^*$ is not T-quasinormal. Then there exists a nonzero multiplier vector $\mu\in\R_+^q$ such that
\[
0\in \sum_{j=1}^q \mu_j\partialT g_j(x^*)+\Normal{\cX}{x^*}
\]
and there exists an infeasible sequence $x^k\to x^*$ satisfying
\[
\mu_j g_j(x^k)>0,\qquad \forall j\in J:=\{j:\mu_j>0\}.
\]
We claim that $\mu_j=0$ for every inactive constraint index $j\notin J(x^*)$. Indeed, if $j\notin J(x^*)$, then $g_j(x^*)<0$. By continuity of $g_j$, one has $g_j(x^k)<0$ for all sufficiently large $k$. Since $\mu_j g_j(x^k)>0$ whenever $\mu_j>0$, this is impossible unless $\mu_j=0$.

Hence the above stationarity relation reduces to
\[
0\in \sum_{j\in J(x^*)}\mu_j\partialT g_j(x^*)+\Normal{\cX}{x^*},
\qquad
\mu_j\ge 0,
\]
with at least one $\mu_j>0$, because $\mu\neq 0$. This contradicts TMFCQ at $x^*$. Therefore $x^*$ is T-quasinormal.
\end{proof}

\begin{lemma}\label{lem:tccp-to-acq-aux}
Let $x^*\in \cX$. Assume that:
\begin{enumerate}[label=(\roman*)]
\item $x^*$ satisfies T--CCP;
\item there exists a neighbourhood $U$ of $x^*$ such that
\[
\widehat N_{\cX}(x)\subset K_T(x),\qquad \forall x\in \cX\cap U,
\]
where $\widehat N_{\cX}(x)$ denotes the Fr\'echet normal cone to $\cX$;
\item
\[
\Normal{\cX}{x^*}=\limsup_{x\to x^*,\,x\in \cX}\widehat N_{\cX}(x).
\]
\end{enumerate}
Then
\[
\Normal{\cX}{x^*}\subset K_T(x^*).
\]
\end{lemma}

\begin{proof}
Let $\omega\in\Normal{\cX}{x^*}$. By the outer-limit representation of the limiting normal cone, there exist sequences $x^k\in \cX$ and $\omega_k\in\widehat N_{\cX}(x^k)$ such that
\[
x^k\to x^*,
\qquad
\omega_k\to \omega.
\]
By assumption (ii), for all sufficiently large $k$ one has
\[
\omega_k\in K_T(x^k).
\]
Since $x^*$ satisfies T--CCP, namely
\[
\limsup_{x\to x^*}K_T(x)\subset K_T(x^*),
\]
it follows from $x^k\to x^*$ and $\omega_k\to\omega$ that
\[
\omega\in K_T(x^*).
\]
Therefore
\[
\Normal{\cX}{x^*}\subset K_T(x^*),
\]
which completes the proof.
\end{proof}

\begin{theorem}\label{thm:tccp-implies-acq}
Assume the hypotheses of Lemma~\ref{lem:tccp-to-acq-aux}. Assume in addition that
\[
K_T(x^*)=L_T(x^*)^\circ,
\qquad
\Normal{\cX}{x^*}=\Tangent{\cX}{x^*}^\circ,
\]
and that $\Tangent{\cX}{x^*}$ is closed and convex. Then TACQ holds at $x^*$, i.e.,
\[
L_T(x^*)\subset \Tangent{\cX}{x^*}.
\]
\end{theorem}

\begin{proof}
Take any $d\in L_T(x^*)$. Suppose, to the contrary, that
\[
d\notin \Tangent{\cX}{x^*}.
\]
Since $\Tangent{\cX}{x^*}$ is closed and convex, by the separation theorem there exists a vector
\[
\omega\in \Tangent{\cX}{x^*}^\circ=\Normal{\cX}{x^*}
\]
such that
\[
\langle \omega,d\rangle>0.
\]
By Lemma~\ref{lem:tccp-to-acq-aux}, we have
\[
\Normal{\cX}{x^*}\subset K_T(x^*).
\]
Hence $\omega\in K_T(x^*)$. Using the assumption
\[
K_T(x^*)=L_T(x^*)^\circ,
\]
we obtain
\[
\omega\in L_T(x^*)^\circ.
\]
Since $d\in L_T(x^*)$, it follows from the definition of the polar cone that
\[
\langle \omega,d\rangle\le 0,
\]
which contradicts $\langle \omega,d\rangle>0$. Therefore our assumption was false, and thus
\[
d\in \Tangent{\cX}{x^*}.
\]
Since $d\in L_T(x^*)$ was arbitrary, we conclude that
\[
L_T(x^*)\subset \Tangent{\cX}{x^*},
\]
that is, TACQ holds at $x^*$.
\end{proof}

\begin{theorem}[T--CCP implies $\mathrm{AKKT}\Rightarrow\mathrm{KKT}$]
\label{thm:tccp-akkt-kkt}
Let $x^*$ be a feasible point of problem~\eqref{prob:P}. Assume that:
\begin{enumerate}[label=(\roman*)]
\item $x^*$ satisfies the KKT-type tangential AKKT condition in the sense of Definition~\ref{def:enhanced-TKKT};
\item $x^*$ satisfies T--CCP in the sense of Definition~\ref{def:TCCP};
\item the multifunction $x\mapsto\partialT f(x)$ is outer semicontinuous at $x^*$.
\end{enumerate}
Then $x^*$ satisfies the enhanced tangential KKT condition.
\end{theorem}

\begin{proof}
Since $x^*$ satisfies the KKT-type tangential AKKT condition, there exist sequences
\[
x^k\to x^*,\qquad
\mu_j^k\ge 0,\qquad
\xi_f^k\in\partialT f(x^k),\qquad
\xi_j^k\in\partialT g_j(x^k),\qquad
\eta_k\in\Normal{\cX}{x^k},
\]
such that
\begin{equation}\label{eq:tccp-proof-1}
\left\|\xi_f^k+\sum_{j=1}^q\mu_j^k\xi_j^k+\eta_k\right\|\to 0,
\end{equation}
and
\begin{equation}\label{eq:tccp-proof-2}
\min\{\mu_j^k,-g_j(x^k)\}\to 0,\qquad j=1,\dots,q.
\end{equation}

Let
\[
J(x^*):=\{j\in\{1,\dots,q\}:\ g_j(x^*)=0\}.
\]
If $j\notin J(x^*)$, then $g_j(x^*)<0$. By continuity of $g_j$, there exists $\varepsilon_j>0$ such that
\[
g_j(x^k)\le -\varepsilon_j
\]
for all sufficiently large $k$. Hence \eqref{eq:tccp-proof-2} implies
\[
\mu_j^k\to 0.
\]
Since each $g_j$ is locally Lipschitz around $x^*$, the sets $\partialT g_j(x)$ are locally bounded near $x^*$. Therefore,
\[
\mu_j^k\xi_j^k\to 0,\qquad \forall j\notin J(x^*).
\]

Define
\[
\omega_k:=\sum_{j\in J(x^*)}\mu_j^k\xi_j^k+\eta_k.
\]
Then from \eqref{eq:tccp-proof-1} we obtain
\[
\xi_f^k+\omega_k\to 0.
\]
By construction, $\omega_k\in K_T(x^k)$ for all $k$.

Since $f$ is locally Lipschitz near $x^*$, the sequence $\{\xi_f^k\}$ is bounded. Passing to a subsequence if necessary, we may assume
\[
\xi_f^k\to \xi_f^*.
\]
By the outer semicontinuity of $x\mapsto\partialT f(x)$, it follows that
\[
\xi_f^*\in\partialT f(x^*).
\]
Hence
\[
\omega_k\to -\xi_f^*.
\]
Since $\omega_k\in K_T(x^k)$ and $x^k\to x^*$, T--CCP yields
\[
-\xi_f^*\in K_T(x^*).
\]
By the definition of $K_T(x^*)$, there exist multipliers $\mu_j^*\ge 0$, $j\in J(x^*)$, together with
\[
\xi_j^*\in\partialT g_j(x^*),\qquad \eta^*\in\Normal{\cX}{x^*},
\]
such that
\[
-\xi_f^*=\sum_{j\in J(x^*)}\mu_j^*\xi_j^*+\eta^*.
\]
Equivalently,
\[
0=\xi_f^*+\sum_{j\in J(x^*)}\mu_j^*\xi_j^*+\eta^*.
\]
Now set $\mu_j^*:=0$ for every $j\notin J(x^*)$. Then we obtain
\[
0=\xi_f^*+\sum_{j=1}^q\mu_j^*\xi_j^*+\eta^*.
\]
Finally, if $j\notin J(x^*)$, then $g_j(x^*)<0$ and $\mu_j^*=0$, so
\[
\mu_j^*g_j(x^*)=0.
\]
If $j\in J(x^*)$, then $g_j(x^*)=0$, so again
\[
\mu_j^*g_j(x^*)=0.
\]
Therefore $x^*$ satisfies the enhanced tangential KKT condition.
\end{proof}

\begin{theorem}[Enhanced tangential KKT under quasinormality]
\label{thm:enhanced-kkt}
Let $x^*$ be a local minimizer of \eqref{prob:P}. Assume the hypotheses of Theorem~\ref{thm:static-fj}, and assume in addition that $x^*$ is T-quasinormal. Then $x^*$ satisfies the enhanced tangential KKT condition.
\end{theorem}

\begin{proof}
By Theorem~\ref{thm:static-fj}, there exist multipliers
\[
\mu_0,\mu_1,\dots,\mu_q\ge 0,
\]
not all zero, and vectors
\[
\xi_f\in\partialT f(x^*),\qquad
\xi_j\in\partialT g_j(x^*),\qquad
\eta\in\Normal{\cX}{x^*}
\]
satisfying
\begin{equation}\label{eq:qn-proof-fj}
0=\mu_0\xi_f+\sum_{j=1}^q\mu_j\xi_j+\eta,
\end{equation}
together with
\begin{equation}\label{eq:qn-proof-comp}
\mu_j g_j(x^*)=0,\qquad j=1,\dots,q.
\end{equation}

We claim that $\mu_0>0$. Suppose, to the contrary, that $\mu_0=0$. Since the multiplier vector is nonzero, there exists some $j$ such that $\mu_j>0$. Let
\[
J:=\{j\in\{1,\dots,q\}:\ \mu_j>0\}.
\]
In the proof of Theorem~\ref{thm:enhanced-tangential-fj}, after extraction of the convergent subsequence leading to $(\mu_0,\mu)$, the complementarity--violation property yields an infeasible sequence $x^k\to x^*$ such that
\[
\mu_j g_j(x^k)>0,\qquad \forall j\in J,\ \forall k,
\]
and each $g_j$, $j\in J$, is proximal subdifferentiable at $x^k$. On the other hand, \eqref{eq:qn-proof-fj} becomes
\[
0\in \sum_{j=1}^q\mu_j\partialT g_j(x^*)+\Normal{\cX}{x^*},
\]
which contradicts T-quasinormality. Therefore $\mu_0>0$.

Dividing \eqref{eq:qn-proof-fj} by $\mu_0$ and setting
\[
\lambda_j:=\frac{\mu_j}{\mu_0},\qquad j=1,\dots,q,
\]
we obtain
\[
0=\xi_f+\sum_{j=1}^q\lambda_j\xi_j+\frac{1}{\mu_0}\eta.
\]
Since $\mu_0>0$, one has $\lambda_j\ge 0$ for all $j$. Moreover, by \eqref{eq:qn-proof-comp},
\[
\lambda_j g_j(x^*)=0,\qquad j=1,\dots,q.
\]
Thus the enhanced tangential KKT condition holds at $x^*$.
\end{proof}

\section{Application to Bilevel Programming }

\noindent
This section applies the tangential optimality theory developed in Sections~3 and~4 to the bilevel programming problem introduced in Section~1. We only recall the model from the introduction and directly derive the tangential optimality systems for its value-function reformulation.

Recall the bilevel problem \eqref{BP}--\eqref{LLP} introduced in Section~1. Let
\[
V(x):=\inf\{f(x,y):\ g_i(x,y)\le 0,\ i=1,\dots,q\}
\]
be the lower-level value function. Then the optimistic bilevel problem can be rewritten in the value-function form
\begin{equation}\label{eq:bilevel-vf}
\begin{aligned}
\min_{x,y}\quad & F(x,y)\\
\text{s.t.}\quad & H(x,y):=f(x,y)-V(x)\le 0,\\
& g_i(x,y)\le 0,\qquad i=1,\dots,q,\\
& G_j(x,y)\le 0,\qquad j=1,\dots,p.
\end{aligned}
\end{equation}
Denote by $\Xi$ the feasible set of \eqref{eq:bilevel-vf}. Throughout this section, we assume that $F,H,g_i,G_j$ are locally Lipschitz and tangentially convex around the reference point.



\begin{theorem}[Enhanced T-Fritz--John condition for bilevel programs]
\label{thm:bilevel-fj}
Let $(\bar x,\bar y)$ be a local minimizer of \eqref{BP}--\eqref{LLP}. 
Assume that $\Tangent{\Xi}{(\bar x,\bar y)}$ is convex and that 
$F,H,g_i,G_j$ are locally Lipschitz and tangentially convex in a neighbourhood of $(\bar x,\bar y)$.

In addition, assume that the value function
\[
V(x):=\inf\{f(x,y):\ g(x,y)\le 0\}
\]
is locally Lipschitz and tangentially convex around $\bar x$.

Then the value-function reformulation \eqref{eq:bilevel-vf} satisfies the sequential enhanced T-Fritz--John type condition at $(\bar x,\bar y)$.
\end{theorem}

\begin{proof}
Problem \eqref{eq:bilevel-vf} is an instance of the general constrained problem \eqref{prob:P} with decision variable $z:=(x,y)$, objective $F$, and inequality constraints $H$, $g_i$, and $G_j$. 

Under the stated assumptions, in particular the tangential convexity and local Lipschitz continuity of the value function $V(x)$, all constraint functions in \eqref{eq:bilevel-vf} satisfy the hypotheses of Theorem~\ref{thm:enhanced-tangential-fj}. 

Therefore, the conclusion follows directly from Theorem~\ref{thm:enhanced-tangential-fj}.
\end{proof}

\begin{remark}
The local Lipschitz continuity of the value function $V(x)$ can be ensured under standard regularity conditions of the lower-level problem; see  \cite{YeZhu1995}\cite{Dempe2002}.
Moreover, it follows from the results of Mashkoorzadeh et al.\cite{Mashkoorzadeh2021} and Gadhi and Ohda\cite{GadhiOhda2024} that $V(x)$ is tangentially convex when the lower-level feasible set is compact and independent of $x$, and the function $f(x,y)$ is uniformly Lipschitz and tangentially convex with respect to $x$.
\end{remark}



\begin{definition}[Bilevel T-quasinormality]
\label{def:bilevel-T-quasinormality}
A feasible point $(\bar x,\bar y)\in \Xi$ is said to be \emph{bilevel T-quasinormal} if the value-function reformulation \eqref{eq:bilevel-vf} is T-quasinormal at $(\bar x,\bar y)$ in the sense of Definition~\ref{def:T-quasinormality}.
\end{definition}

\begin{theorem}[Enhanced T-KKT condition for bilevel programs]
\label{thm:bilevel-kkt}
Let $(\bar x,\bar y)$ be a local minimizer of \eqref{BP}--\eqref{LLP}. Assume that:
\begin{enumerate}[label=(\roman*)]
\item the assumptions of Theorem~\ref{thm:bilevel-fj} hold;
\item $F,H,g_i,G_j$ are Clarke regular at $(\bar x,\bar y)$;
\item the normal cone mapping is outer semicontinuous at $(\bar x,\bar y)$;
\item $(\bar x,\bar y)$ is bilevel T-quasinormal.
\end{enumerate}
Then the value-function reformulation \eqref{eq:bilevel-vf} satisfies the enhanced T-KKT condition at $(\bar x,\bar y)$.
\end{theorem}

\begin{proof}
Applying Theorems~\ref{thm:static-fj} and \ref{thm:enhanced-kkt} to the single-level reformulation \eqref{eq:bilevel-vf}, we obtain the desired enhanced T-KKT system.
\end{proof}

\section*{Declarations}

\noindent\textbf{Funding.}
No funds, grants, or other support was received.

\smallskip

\noindent\textbf{Competing Interests.}
The authors declare no competing interests.

\smallskip

\noindent\textbf{Data Availability.}
No datasets were generated or analyzed during the current study.

\bibliography{refs}

\end{document}